\documentclass[11pt,oneside,a4paper,reqno]{amsart}

\usepackage[a4paper,width=160mm,top=25mm,bottom=30mm]{geometry}
\usepackage{amsmath}
\usepackage{amssymb}
\usepackage{amsthm}
\usepackage{mathrsfs}
\usepackage{mathtools}
\mathtoolsset{showonlyrefs}

\usepackage{graphicx}
\usepackage{float}
\usepackage{caption}
\usepackage[utf8]{inputenc}
\usepackage[T1]{fontenc}
\usepackage[english]{babel}

\usepackage{enumerate}
\usepackage[inline]{enumitem}
\setlist[enumerate]{leftmargin=*, label=\arabic*.}
\usepackage{xcolor}
\usepackage{algorithm}
\usepackage{algpseudocode}
\usepackage{todonotes}
\setuptodonotes{inline}
\usepackage{hyperref}
\usepackage{footnote}
\makesavenoteenv{abstract}

\usepackage[
  style=alphabetic,
  doi=false,
  isbn=false,
  url=false,
  eprint=false
]{biblatex}
\theoremstyle{plain}
\newtheorem{theorem}{Theorem}[section]
\newtheorem{corollary}[theorem]{Corollary}
\newtheorem{proposition}[theorem]{Proposition}

\newtheorem{lemma}[theorem]{Lemma}

\theoremstyle{remark}
\newtheorem{remark}[theorem]{Remark}

\theoremstyle{definition}

\DeclareFieldFormat*{title}{\mkbibemph{#1}}   
\DeclareFieldFormat{journaltitle}{#1\isdot}  
\DeclareFieldFormat{booktitle}{#1\isdot}     
\DeclareFieldFormat{maintitle}{#1\isdot}     
\DeclareFieldFormat{issuetitle}{#1\isdot}

\graphicspath{{./Images/}}

\title{\bfseries
Counting Minimal Vertex Cutsets and a Gap at 1 for Site Percolation on Vertex-Transitive Graphs
}
\author{Joel Bassil}
\begin{document}

\begin{abstract}
We consider site percolation on general graphs and give a sufficient condition on the escape probability of a certain random walk with degree-dependent conductances, such that the number of minimal vertex cutsets of size $n$ separating a given vertex from infinity is bounded above exponentially in $n$. This result extends Theorem~2 of \cite{easoCountingMinimalCutsets2025} under slightly stronger assumptions. Moreover, we give an alternative sufficient condition in terms of the isoperimetric dimension or, more specifically, an isoperimetric-type inequality.

Furthermore, our theorem is sufficient to extend the results of Panagiotis and Severo \cite{panagiotisGap1Percolation2023}, and show that there exists a universal positive constant $\varepsilon_1$ such that site percolation on every infinite, connected, locally finite, vertex-transitive graph satisfies $p_c = 1$ or $p_c \leq 1-\varepsilon_1$.
\end{abstract}

\maketitle
\markboth{JOEL BASSIL}{JOEL BASSIL}

\vspace{-2.5em}
\renewcommand{\abstractname}{AI Use Statement}    

\section{Introduction}\label{sec:intro}

Let $G = (V, E)$ be an infinite, connected, locally finite graph. A set of vertices $F \subset V$ is called a \textit{vertex cutset} from a vertex $v$ to $\infty$ if $v$ belongs to a finite connected component of the graph induced by $V\setminus F$, which we shall abbreviate by $G\setminus F$. A vertex cutset is called \textit{minimal} if no proper subset of it is a vertex cutset. Let $\mathcal{Q}_n^\mathrm{site}(v)$ be the set of minimal vertex cutsets from $v$ to $\infty$ of cardinality $n$ and consider the quantities
\[
q_n^\mathrm{site} := \sup_{v \in V} |\mathcal{Q}_n^\mathrm{site}(v)|, \quad \kappa^\mathrm{site}(G) := \sup_{n \geq 1} {q^\mathrm{site}_n}^{1/n}.
\]
These quantities appear naturally in the study of percolation where they play an important role in proving non-triviality of the phase transition; for example, within the Peierls argument \cite{Peierls_1936}. As such, we are interested in cases where $\kappa^\mathrm{site}(G) < \infty$.

Consider $G$ as a network $(V,E,c)$ with conductances $c:=(c(x,y))_{\{x,y\}\in E}$ where, for $x\sim y$,
\begin{equation}\label{conductances}
c(x,y):=1/\max(\deg(x),\deg(y)),
\end{equation}
and define the total conductance at a vertex $x$ to be $\pi(x):=\sum_{y\sim x}c(x,y)$, which is at most 1. Let $(X_t)_{t\in \mathbb{N}_0}$ be the discrete-time (lazy) random walk defined by the transition probabilities $p(x,y):=c(x,y)$ for $y\sim x$ and $p(x,x):= 1-\pi(x)$, and let $\mathbb{P}^X$ be its law with corresponding expectation $\mathbb{E}^X$. Note that throughout this paper, $(X_t)_{t\in \mathbb{N}_0}$ refers only to the random walk on the network of $G$ with these specific conductances. Our first theorem is the following.

\begin{theorem} \label{thm2_extend}
For every $\varepsilon\in(0,1)$ there exists $K:=K(\varepsilon)>0$ such that the following holds. Let $G$ be an infinite, connected, locally finite graph such that for all vertices $v\in V$,
\[
\mathbb{P}^X_v(\forall t \geq 1: X_t\neq v)\geq \varepsilon.
\]
Then $\kappa^{\mathrm{site}}(G) \leq K< \infty$.
\end{theorem}

The proof follows the same approach sketched in \cite[Section~9]{easoCountingMinimalCutsets2025}, which involves prescribing the exposed vertex boundary of a connected component of the Gaussian Free Field (GFF) on $G$, and lower bounding the probability of such an event, before counting over the set of minimal vertex cutsets. See Section~\ref{sec:preliminaries} for the required background on the GFF.

Let now $(Z_t)_{t\in \mathbb{N}_0}$ be the jump chain of $(X_t)_{t\in \mathbb{N}_0}$ and let $\mathbb{P}^Z$ be its law with corresponding expectation $\mathbb{E}^Z$. We note that, by decomposing on the first step, 
\begin{equation}\label{first_step}
\mathbb{P}^X_v(\forall t \geq 1: X_t\neq v) = \pi(v)\cdot \mathbb{P}^Z_v(\forall t \geq 1: Z_t\neq v) = \mathcal{C}(v\leftrightarrow  \infty),
\end{equation}
where $\mathcal{C}$ denotes the effective conductance on the network of $G$ with conductances as in \eqref{conductances}. As above, throughout this paper, the effective conductance $\mathcal{C}$ should always be interpreted this way. Hence, the assumption in Theorem~\ref{thm2_extend} is equivalent to $\mathcal{C}(v\leftrightarrow  \infty) \geq \varepsilon$.

As mentioned in the abstract, Theorem~\ref{thm2_extend} is an extension of \cite[Theorem~2]{easoCountingMinimalCutsets2025} (under stronger assumptions), where the authors obtain an analogous result for bond cutsets. They assume uniform transience of the simple random walk (SRW) on $G$, whereas by \eqref{first_step}, our assumption is uniform transience of the network of $G$ with conductances $c$. That is, they assume that there exists $\varepsilon > 0$ such that for all vertices $v\in V$,
\[
d_v \mathbb{P}^\mathrm{SRW}_v(\forall t \geq 1 : X_t \neq v) \geq \varepsilon.
\]

It turns out that the exact extensions of both theorems of \cite{easoCountingMinimalCutsets2025} to site percolation are false; one can construct counterexamples by modifying a rooted binary tree and replacing edges with gadgets that allow many choices for where to block (in constructing a cutset) one of two local routes, while keeping every vertex uniformly close to a supercritical, uniformly transient shortcut skeleton \cite{bassilCounterexamplesSitepercolationAnalogues2026}. The assumptions of our theorem are thus necessarily stronger; the exact reasons for which are elucidated in context in Remark~\ref{strenthened_assump}. Indeed, our assumptions upper bound the total conductance out of each vertex by $1$, whereas the assumptions of \cite{easoCountingMinimalCutsets2025} mean the total conductance out of each vertex is dependent on its degree. Qualitatively, the random walk $(X_t)_{t\in \mathbb{N}_0}$ is more likely to avoid high-degree vertices, whereas the SRW on $G$ has no such bias. Despite these strengthened assumptions, our theorem suffices for the purpose of our first corollary. To state it, we must first define what we mean by site percolation and its critical parameter.

Let $\mathbb{P}_p$ denote Bernoulli site percolation of parameter $p \in [0,1]$ on $G$, where each vertex is open with probability $p$ independently of the other vertices. An edge is present in a percolation configuration if both of its endpoints are open. For a vertex $v$ and percolation configuration $\omega$, we write $C_v(\omega)$ for the open cluster of $v$ in $\omega$, or just $C_v$ when the configuration is understood. 

Consider the percolation probabilities $\theta_v(p) := \mathbb{P}_p(v \leftrightarrow \infty)$. Here, $\{v \leftrightarrow \infty\}$ denotes the event that $v$ belongs to an infinite open connected component. We may also write $\{v \stackrel{A}{\leftrightarrow} u\}$ to denote the event that there exists a path of open vertices in a subset $A\subset V$ between $u$ and $v$. The critical percolation parameter is defined as
\[
p_c:=p_c(G) := \inf\{p \in [0,1] : \theta_v(p) > 0\},
\]
where we recall that for any $u,v\in V$, $\theta_u(p) > 0 \iff \theta_v(p) > 0$ so this definition is independent of $v$. 

The following corollary, of independent yet significant interest, extends results by Panagiotis and Severo \cite{panagiotisGap1Percolation2023} who prove the gap-at-1 phenomenon for site percolation on Cayley graphs of superlinear growth (see Theorem~\ref{cayley_gap} below for a restatement). Their proof utilises a comparison between Bernoulli percolation and the GFF, and decomposes the GFF into independent finite-range fields before integrating each one out and compensating by a slight increase in the percolation parameter. Our proof side-steps the Cayley-graph-specific machinery and instead follows the idea of \cite[Section~2, Note~6]{easoCountingMinimalCutsets2025} which is outlined below. During the final write-up of the present paper, this corollary was also proved independently by \cite{hartarskyWeightedIsoperimetryImplies2026} via a different method - studying bond percolation on weighted graphs. It follows immediately from their Theorem~1.7.

\begin{corollary}[Gap at 1 for site percolation on vertex transitive graphs]\label{gapat1}
There exists a universal constant $\varepsilon_1>0$ such that $p_c=1$ or $p_c\leq 1-\varepsilon_1$ for any infinite, connected, locally finite, vertex-transitive graph $G$.
\end{corollary}

The proof of this result is split into the cases of whether $G$ is recurrent or transient. The former case follows the proof of \cite[Theorem~1.7]{hutchcroftNontrivialityPhaseTransition2024} and uses the analogous gap-at-1 result for Cayley graphs \cite[Theorem~1.1]{panagiotisGap1Percolation2023}; the latter case follows the proof of Theorem~\ref{thm2_extend} and the classical Peierls union bound argument.

A more concise way to state the assumptions of many theorems involving (uniform) transience, and indeed, Theorem~\ref{thm2_extend}, is in terms of the isoperimetric dimension, the definition of which we recall now. We define the isoperimetric dimension of the graph $G=(V,E)$ with conductances $c$ by
\begin{equation} \label{isodim}
\mathrm{Dim}(G;c) := \sup \left\{ d \geq 1 : \inf_{\substack{K \subset V \\ 0 < |K| < \infty}} \frac{|\partial^E K|_{c}}{|K|_\pi^{(d-1)/d}} > 0 \right\},
\end{equation}
where $|\partial^E K|_{c}$ is the sum of the conductances of the edges in the edge-boundary of $K$ and $|K|_\pi$ is the sum of the $\pi$-measure of the vertices in $K$. It is well-known that an isoperimetric dimension greater than 2 implies uniform transience, so the following corollary is to be expected.

\begin{corollary}\label{dim2}
Let $G$ be an infinite, connected, locally finite graph with conductances $c$ as in \eqref{conductances}. Then
\[
\mathrm{Dim}(G;c)>2 \implies \kappa^{\mathrm{site}}(G) < \infty.
\]
In particular, if there exist $\beta > 0$ and $d>2$ such that $|\partial^E K|_c \geq \beta |K|_\pi^{(d-1)/d}$ for all finite $K\subset V$, then there exists $M:=M(\beta,d)$ such that $\kappa^{\mathrm{site}}(G) \leq M < \infty$.
\end{corollary}

The isoperimetric inequality translates to a finite upper bound on effective resistance due to, for example, \cite[Theorem~6.41]{lyonsProbabilityTreesNetworks2017}, which by \eqref{first_step} is equivalent to the assumption of Theorem~\ref{thm2_extend}. The result follows from said theorem's proof.

\subsection*{Acknowledgements} The present paper, based heavily on \cite{easoCountingMinimalCutsets2025} in terms of both results and presentation, is a continuation of the author's undergraduate research project at the University of Bath. The author is very grateful to Dr Christoforos Panagiotis for supervising said project, and prompting the investigations into extensions of the results of \cite{easoCountingMinimalCutsets2025}.

\subsection*{AI Use Statement} A discussion with the LLM `ChatGPT-5.6 sol Pro' was developed into the proof of Proposition~\ref{harmonic_cover}. The LLM provided the initial idea to bound the length of the greedily constructed set via bounding the effective conductance increments.

\section{Background, definitions and notation}\label{sec:preliminaries}

In this section, we fix $G=(V,E)$ a locally finite, connected graph.

\subsection*{Vertex boundary} Given a set of vertices $S\subset V$, we define the vertex boundary $\partial^VS$ to be the set of vertices $\{u\in V\setminus S: u \text{ has a neighbour in } S \}$.

\subsection*{Exposed vertex boundary} For a finite, non-empty set of vertices $S\subset V$, we define its exposed vertex boundary $\partial_\infty^VS$ to be the set of vertices
\[
\{v\in\partial^VS: v \leftrightarrow \infty \text{ in } G\setminus (S\cup (\partial^VS\setminus \{v\})) \}.
\]

\subsection*{Ray} A ray is an infinite sequence of vertices in which each vertex appears at most once in the sequence and each two consecutive vertices in the sequence are the two endpoints of an edge in the graph. In the context of a percolation configuration, the use of this term implicitly assumes that all of the relevant vertices are open.

\subsection*{Increasing events and positive association} A percolation event $\mathcal{E}$, being a measurable subset of $\{0,1\}^E$, is called increasing if for all
percolation configurations $\omega, \xi\in \{0,1\}^E$ satisfying $\omega \leq \xi$
for the standard product (partial) ordering, we have
$\omega \in \mathcal{E} \implies \xi \in \mathcal{E}$. A percolation measure $\mathbb{P}$ is said to be positively associated if 
\[
\mathbb{P}(A \cap B) \geq \mathbb{P}(A)\mathbb{P}(B)
\]
for all increasing events $A,B$. This is often called the \textit{FKG-inequality}; required background can be found in \cite{grimmettPercolation1999}, for example.

\subsection*{Gaussian Free Field (GFF)} Assume now that $G = (V,E)$ is infinite, and consider it as the network $(V,E,c)$ with conductances as in \eqref{conductances}. As in \cite[Section~1.1]{berestyckiGaussianFreeField2025}, for $x,y \in V$ we write
\[
G(x,y) := \mathbb{E}^X_x\left(\sum_{t=0}^{\infty} \mathbf{1}_{\{X_t = y\}}\right)
\]
for the Green function associated to $(X_t)_{t\in \mathbb{N}_0}$, which is finite whenever $(X_t)_{t\in \mathbb{N}_0}$ is transient. Moreover, for $A \subseteq V$ and a vertex $x \in V$, write
\[
T_A := \inf\{t \geq 0 : X_t \in A\}, \qquad T_x^+ := \inf\{t \geq 1 : X_t = x\}
\]
for the hitting time of $A$ and the first return time to $x$, respectively. Then, for $\Lambda \subset V$, we write
\[
G_\Lambda(x,y) := \mathbb{E}^X_x\left(\sum_{t < T_\Lambda} \mathbf{1}_{\{X_t = y\}}\right)
\]
for the Green function associated to $(X_t)_{t\in \mathbb{N}_0}$ killed upon entering $\Lambda$. Both kernels are non-negative and killing the walk can only decrease them, so
\begin{equation}\label{eq:greenkilled}
0 \le G_\Lambda(x,y) \leq G(x,y),
\end{equation}
for all $x,y \in V$. In addition, $G_\Lambda(x,x) \ge 1$ for every $x \notin \Lambda$, by the contribution of the term $t = 0$.

Suppose now that $(X_t)_{t\in \mathbb{N}_0}$ is transient. The (centred) Gaussian Free Field on $G$ is the centred Gaussian field $\varphi = (\varphi(x))_{x \in V}$ with covariance $G(x,y)$ for $x,y \in V$. Given $\Lambda \subset V$, the GFF $\varphi_{\Lambda}$ on $V \setminus \Lambda$ with zero boundary conditions on $\Lambda$ is the (centred) Gaussian field with covariance $G_\Lambda$. By \cite{pittPositivelyCorrelatedNormal1982}, the GFF restricted to a finite set of vertices is positively associated.

\section{Preliminaries}\label{exp_bound_con_ineq}
In this section, we fix $G=(V,E)$ an infinite, locally finite, connected graph. The first two lemmas we require provide sufficient conditions to prescribe an exposed vertex boundary to be a chosen minimal vertex cutset. This is similar in presentation to Lemmas~1 and 2 in \cite{easoCountingMinimalCutsets2025}, although slightly different due to exposed boundaries not being defined exactly analogously.

\begin{lemma}\label{first_bound_pres}
Let $\Pi$ be a minimal vertex cutset separating a vertex $o$ from $\infty$, and let $A$ be the connected component of $o$ in $G \setminus \Pi$. Then
$$\partial^VA = \Pi.$$
Moreover, for every $\pi \in \Pi$ there exists a ray $(v_0, v_1, \ldots)$, with $v_0 = \pi$, such that
$$v_j \notin A \cup \Pi \qquad \forall j \geq 1.$$
\end{lemma}

\begin{proof}
If $v \in \partial^VA \setminus \Pi$, then $v$ is present in $G \setminus \Pi$ and adjacent to $A$, contradicting that $A$ is a connected component. Hence $\partial^VA \subseteq \Pi$.

Now fix $\pi \in \Pi$. By minimality, the component of $o$ in $G \setminus (\Pi \setminus \{\pi\})$ is infinite and therefore contains a ray from $o$. This ray avoids $\Pi \setminus \{\pi\}$ and, since $A$ is finite, must leave $A$ through $\pi$. Thus $\pi \in \partial^VA$, proving the remaining inclusion and thus that $\partial^VA = \Pi$.

Take the tail of the ray, beginning at the first visit to $\pi$. It avoids $\Pi \setminus \{\pi\}$. If it entered $A$ after its first visit to $\pi$, then, because $A$ is finite and vertices in a ray are not repeated, it would later have to leave $A$ through $\Pi \setminus \{\pi\}$. This is impossible.
\end{proof}

\begin{lemma}\label{second_bound_pres}
Let $\Pi$ be a minimal vertex cutset separating a vertex $o$ from $\infty$, and let $A$ be the connected component of $o$ in $G \setminus \Pi$. If $\varnothing \neq S \subseteq A$ and every $\pi \in \Pi$ has a neighbour in $S$, then
$$\partial_\infty^V S = \Pi.$$
\end{lemma}

\begin{proof}
The hypotheses and Lemma~\ref{first_bound_pres} give
$$\Pi \subseteq \partial^VS \subseteq A \cup \Pi.$$
Indeed, every $\pi \in \Pi$ has a neighbour in $S$ and $\pi \notin A \supset S$, giving the first inclusion. Moreover, if $v\notin S$ but $v\sim u$ for some $u \in S$ then either $v\in A$, or $v\notin A$ and since $u\in A$ also, $v\in \partial^VA = \Pi$ by Lemma~\ref{first_bound_pres}.

Now, fix $\pi \in \Pi$ and use the ray supplied by Lemma~\ref{first_bound_pres}. It avoids $A \cup (\Pi \setminus \{\pi\})$, hence avoids $S \cup (\partial^VS \setminus \{\pi\})$. Therefore $\pi \in \partial_\infty^V S$.

Conversely, let $v \in \partial^VS \setminus \Pi$ (this is enough, since $\partial_\infty^V S \subset \partial^V S$). Then $v \in A$. Every path from $v$ to infinity must leave $A$ through $\Pi$, but $\Pi \subseteq \partial^VS \setminus \{v\}$ so there cannot exist a path from $v$ to $\infty$ in $G\setminus (S\cup (\partial^VS\setminus \{v\}))$. Hence, by definition, $v \notin \partial_\infty^V S$.
\end{proof}

It is well-known that the GFF admits a Markov property, in the sense that one may condition on the values in a certain set of vertices and decompose the GFF as the sum of a GFF on the remaining vertices and a deterministic harmonic extension into said vertices. The following lemma states this for infinite, transient graphs, and it follows from the same result for finite graphs endowed with a killing measure \cite[Proposition~2.3]{Sznitman2012} and a well-known exhaustion argument.

\begin{lemma}[Markov property of the GFF]\label{markov}
Suppose that $(X_t)_{t\in \mathbb{N}_0}$ is transient and let $\varphi$ be the GFF on $G$ associated to $(X_t)_{t\in \mathbb{N}_0}$. Let $\Lambda \subset V$ be finite and non-empty. Then we may write
\begin{equation}\label{eq:markov}
\varphi = \varphi_{\Lambda} + h_\Lambda ,
\end{equation}
where $\varphi_{\Lambda}$ is a centred GFF on $V \setminus \Lambda$ with zero boundary conditions on $\Lambda$, which is independent of $\sigma(\varphi(v):v \in \Lambda)$. Moreover, $h_\Lambda$ is measurable with respect to $\sigma(\varphi(v) : v \in \Lambda)$ and is given by
\begin{equation}\label{eq:hK}
h_\Lambda(x) = \mathbb{E}^X_x\big(\varphi(X_{T_\Lambda}) \mathbf{1}_{\{T_\Lambda < \infty\}}\big) = \sum_{v \in \Lambda} \mathbb{P}^X_x\big(T_\Lambda < \infty,\, X_{T_\Lambda} = v\big)\, \varphi(v)
\end{equation}
for every $x \in V$ \footnote{Although it is easy to verify that this expression for $h_\Lambda(x)$ is indeed harmonic off $\Lambda$ and agrees with $\varphi$ on $\Lambda$, such an expression need not be unique. For example, on $\mathbb{Z}^3$ with $\Lambda = \{0\}$ and prescribed value $1$ at the origin, both $h \equiv 1$ and $h(x) = \mathbb{P}_x(T_0 < \infty)$ are harmonic off $\Lambda$ and equal to $1$ on $\Lambda$.}.
\end{lemma}

It is often the case that the finite set to which we wish to apply the Markov property is not deterministic but random. The following lemma allows us to do so given that the set is a \textit{stopping set}, the spatial analogue of a stopping time. It follows from Lemma~\ref{markov} and a standard argument, decomposing over the countably many possibilities of $\Lambda$.

\begin{lemma}[Strong Markov property of the GFF]\label{strong_markov}
Let $G$ and $\varphi$ be as in Lemma~\ref{markov}, and let $\Lambda \subset V$ be a finite, non-empty, possibly random set. Assume also that $\Lambda$ is a stopping set, in the sense that $\{\Lambda = \lambda\} \in \sigma(\varphi (v): v\in \lambda)$ for every finite $\lambda \subset V$. Then \eqref{eq:markov} and \eqref{eq:hK} hold conditionally on $\mathcal{H} = \sigma\big(\Lambda, \{\varphi(v) : v \in \Lambda\}\big)$.
\end{lemma}

In the proof of Theorem~\ref{thm2_extend}, we will make use of the following proposition which infers an exponential lower bound on the probability of a percolation cluster (of a chosen vertex) intersecting every vertex of a set which is `easy' to connect to.

\begin{proposition}[{\cite[Proposition~3]{easoCountingMinimalCutsets2025}}]\label{prop3}
Let $G'=(V',E')$ be a finite, connected graph. Let $\mathbb{P}$ be a positively associated percolation measure on $G'$. Let $B\subset V'$, let $\theta,p\in(0,1]$ and suppose that $\mathbb{P}(u\leftrightarrow B)\geq \theta$ for every $u\in V'$, and $\mathbb{P}(e \text{ is open})\geq p$ for every $e\in E'$. Then for every $o\in V'$,
\[
\mathbb{P}\left(\bigcap_{b\in B}\{o\leftrightarrow b\}\right) \geq c^{|B|},
\]
where $c:=\left(\dfrac{p\theta}{2}\right)^{3/\theta}$.
\end{proposition}

The following results concerning Cayley and quotient graphs are used in the proof of the recurrent case of Corollary~\ref{gapat1}; we include the proof of one as it contains a slight modification to the original hypotheses, in that we consider site percolation instead of bond percolation.

\begin{theorem}[{\cite[Theorem~1.1]{panagiotisGap1Percolation2023}}]\label{cayley_gap}
There is a universal constant $\varepsilon>0$ such that $p_c(G')\leq1-\varepsilon$ for site percolation on any Cayley graph $G'$ of superlinear growth.
\end{theorem}

Let $H\leq\mathrm{Aut}(G)$. For $v\in V$, write $Hv:=\{hv : h\in H\}$ for the orbit of $v$ under $H$. The quotient graph $G/H$ is the simple graph with vertex set $V/H:=\{Hv : v\in V\}$, in which two distinct orbits $Hv, Hv'$ are joined by an edge whenever some representatives $x\in Hv$ and $y\in Hv'$ satisfy $x\sim y$ in $G$.

\begin{theorem}[{\cite[Theorem~1]{benjaminiPercolation$Z^d$Many1996}}]\label{benj}
Let $H\leq\mathrm{Aut}(G)$. Then $p_c(G) \leq p_c(G/H)$ for site percolation.
\end{theorem}

Write $\tau:V\to V/H$, $\tau(v):=Hv$, for the quotient map. In particular, $\tau^{-1}(\tau(v))=Hv$ is precisely the orbit of $v$.

\begin{lemma}[{A site percolation/vertex-orbit version of \cite[Lemma~2.9]{hutchcroftNontrivialityPhaseTransition2024}}] \label{groups}
Suppose that $H \leq\mathrm{Aut}(G)$ satisfies $|Hv| \leq k$ for some $k \geq 1$ and every $v \in V$. For each $v \in V$ and $p \in [0,1]$, the cluster of $\tau(v)$ in Bernoulli-$(1-(1-p)^k)$ site percolation on $G/H$ stochastically dominates the image under $\tau$ of the cluster of $v$ in Bernoulli-$p$ site percolation on $G$. That is,
\[
\mathbb{P}^{G/H}_{1-(1-p)^k}\!\left(C_{\tau(v)} \in A\right) \geq 
\mathbb{P}^{G}_{p}\!\left(\tau(C_v) \in A\right)
\]
for every increasing measurable set $A \subseteq \{0,1\}^{V/H}$, and where by writing $S\in A$ for some $S\subset V/H$, we mean $\mathbf{1}_S\in A$.
\end{lemma}

\begin{proof}
Let $\omega$ be Bernoulli-$p$ site percolation on $G$, and let $\eta\in\{0,1\}^{V/H}$ be defined by taking $\eta(u)=1$ if and only if there is at least one $\omega$-open vertex in the preimage $\tau^{-1}(u)$. For each vertex $v$ of $G$, we clearly have that the image under $\tau$ of the cluster of $v$ in $\omega$ is contained in the cluster of $\tau(v)$ in $\eta$. Moreover, the vertices of $G/H$ are open or closed in $\eta$ independently of one another and each is open with probability at most $q:=1-(1-p)^k$.

Let $(U(u))_{u\in V/H}$ be i.i.d.\ Uniform$[0,1]$ random variables with joint law $\mathbb{P}$, and let \linebreak${p_u:=\mathbb{P}^G_p(\eta(u)=1)}$. Define $\zeta(u):=\mathbf{1}_{\{U(u)<q\}}$; then $\zeta$ is Bernoulli-$q$ site percolation on $G/H$. Also define $\tilde\eta(u):=\mathbf{1}_{\{U(u)<p_u\}}$; since the $\tilde\eta(u)$ are independent across $u$ and each is Bernoulli-$p_u$, $\tilde\eta$ under law $\mathbb{P}$ has the same distribution as $\eta$ under law $\mathbb{P}_p^G$. Since $p_u\le q$, we have $\{U(u)<p_u\}\subseteq\{U(u)<q\}$, so $\tilde\eta\leq\zeta$ pointwise. Since $\tau(C_v)\subseteq C_{\tau(v)}$ and $A$ is increasing, 
\begin{align*}
\mathbb{P}^G_p(\tau(C_v)\in A)\leq\mathbb{P}^G_p(C_{\tau(v)}(\eta)\in A) &= \mathbb{P}(C_{\tau(v)}(\tilde\eta)\in A) \\ &\leq\mathbb{P}(C_{\tau(v)}(\zeta)\in A) = \mathbb{P}^{G/H}_{1-(1-p)^k}\!\left(C_{\tau(v)} \in A\right). \qedhere
\end{align*}
\end{proof}

We also state one more result, to be used in the proof of Corollary~\ref{dim2} as previously mentioned, which gives a bound on the effective resistance from a finite set to infinity. First, we must define two more structural quantities of interest. Define the isoperimetric profile $\varphi(G, t;c)$ by
\[
\varphi(G, t;c) := \inf_{K \subset V}\{|\partial^E K|_c ; t \leq |K|_\pi < \infty\}.
\]

Define also the \textit{anchored} isoperimetric profile $\varphi(G, A, t;c)$ by
\[
\varphi(G, A, t;c) := \inf_{K \subset V}\{|\partial^E K|_c ; A \subseteq K, K/A \text{ is connected}, t \leq |K|_\pi < \infty\}
\]
for $A \subset V(G)$. Here, we say that $K/A$ is connected when the graph induced by $K$ in $G/A$ is connected, where we have identified all of $A$ to a single vertex.

\begin{theorem}[{\cite[Theorem~6.41]{lyonsProbabilityTreesNetworks2017}}]\label{lp6.41}
View $G$ as a network with arbitrary edge conductances and suppose that $|V(G)|_\pi = \infty$. Let $A$ be a finite set of vertices. Suppose also that $\varphi(t) := \varphi(G, A, t)$ has the property that $\varphi(t) \geq f(t)$ for some increasing function $f$ on $[|A|_\pi, \infty)$ that satisfies $0 < f(t) \leq t$ and $f(2t) \leq \alpha f(t)$ for some $\alpha$. Then
\[
\mathscr{R}(A \leftrightarrow \infty) \leq \int_{|A|_\pi}^\infty \frac{4\alpha^2}{f(t)^2} dt.
\]
\end{theorem}

\section{A small, well-connected, harmonic-measure covering set}\label{sec:covering_set}

In this section, we show that if the random walk $(X_t)_{t\in \mathbb{N}_0}$ (from Section~\ref{sec:intro}) can easily escape from any vertex, then any finite connected set $A$ can be covered in harmonic measure by a small collection of vertices (found greedily), proportional to the size of $A$'s vertex boundary. We will apply this in the next section to prove Theorem~\ref{thm2_extend}.

\begin{proposition}\label{harmonic_cover}
Let $G$ be an infinite, connected, locally finite graph. Assume also that there exists $\varepsilon>0$ such that for all vertices $v\in V$, $\mathbb{P}^X_v(\forall t \geq 1: X_t\neq v)\geq \varepsilon$. Let $A$ be a finite, connected set in $V(G)$ with vertex boundary $\Pi$. Then, for any vertex $o\in A$, there exists a set $Y\subset A$ containing $o$ such that, for all $v\in A$, $\mathbb{P}^X_v(\mathrm{T}_Y<\mathrm{T}_\Pi)\geq 1/2$ and $|Y|<4|\Pi|/\varepsilon$.
\end{proposition}

\begin{proof}
We define a sequence of nested sets in the following way. Let $Y_1 := \{o\}$. Having defined $Y_1 \subset Y_2 \subset \cdots \subset Y_i$, define
\[
W_i := \left\{v \in A : \mathbb{P}^X_v(T_{Y_i} < T_\Pi) \geq \frac{1}{2}\right\}.
\]
Note that $W_i$ is non-empty for all $i$ since $o \in Y_i$ implies $o \in W_i$. If $W_i \neq A$, then, since $A$ is connected, there exists an edge $\{u,v\} \in E$ with $u \in W_i$ and $v \notin W_i$; let $Y_{i+1} := Y_i \cup \{v\}$ and $y_{i+1} := v$. Otherwise, if $W_J = A$ for some $J$, take $Y := Y_J$. Note that $J < \infty$ since $A$ is finite.

Consider now the effective conductances
\begin{equation}\label{eq:cond_def}
\mathcal{C}_i := \mathcal{C}(Y_i, \Pi) := \sum_{y \in Y_i} \sum_{z \sim y} c(y,z)\mathbb{P}^X_z(T_\Pi < T_{Y_i}).
\end{equation}
Write $q_i(z) := \mathbb{P}^X_z(T_{Y_i} < T_\Pi)$ and $r_i(z) := \mathbb{P}^X_z(T_{y_{i+1}} < T_{Y_i \cup \Pi})$, so that
\[
q_{i+1}(z) = \mathbb{P}^X_z(T_{Y_{i+1}} < T_\Pi) = q_i(z) + r_i(z)\big(1 - q_i(y_{i+1})\big),
\]
since either $Y_i$ is hit before $\Pi$, or $y_{i+1}$ is hit before $Y_i$ and $\Pi$, and then $\Pi$ is hit before $Y_i$, starting afresh from $y_{i+1}$ and using the strong Markov property. Define also the symmetric bilinear form
\[
\mathcal{E}(f,g) := \sum_{\{x,y\} \in E} c(x,y) \big(f(x) - f(y)\big)\big(g(x)-g(y)\big),
\]
commonly called the \textit{energy}, or \textit{Dirichlet energy}, where the sum is over unordered edges. Write $\mathcal{E}(f) := \mathcal{E}(f,f)$. It is a standard fact (often called the \textit{Dirichlet principle}; see, for example, \cite[Exercise~1.3.11]{doyleRandomWalksElectric2000}) that $\mathcal{E}(q_i) = \mathcal{C}_i$. However, we shall include the proof here due to its brevity and the usefulness of some intermediate expressions:
\begin{align*}
\mathcal{E}(f,g)
  &= \frac12 \sum_{x\in V}\sum_{y\sim x} c(x,y)\big(f(x)-f(y)\big)\big(g(x)-g(y)\big) \\
  &= \frac12 \sum_{x\in V}\sum_{y\sim x} c(x,y)\,f(x)\big(g(x)-g(y)\big)
   - \frac12 \sum_{x\in V}\sum_{y\sim x} c(x,y)\,f(y)\big(g(x)-g(y)\big) \\
  &= \sum_{x\in V} f(x) \sum_{y\sim x} c(x,y)\big(g(x)-g(y)\big),
\end{align*}
by noting that the conductances are symmetric and symmetrising the sum over edges. Then since $q_i$ is harmonic off $Y_i \cup \Pi$, equal to $1$ on $Y_i$, and $0$ on $\Pi$, we have that
\[
\mathcal{E}(q_i) = \sum_{x \in Y_i} \sum_{z \sim x} c(x,z)\big(1 - q_i(z)\big) = \mathcal{C}_i,
\]
since for $x\notin Y_i \cup \Pi$, $\sum_{z\sim x}c(x,z)(q_i(x)-q_i(z)) = \pi(x)\,q_i(x) - \sum_{z\sim x} c(x,z)q_i(z)$, and harmonicity gives $q_i(x) = \sum_{z\sim x} \frac{c(x,z)}{\pi(x)}q_i(z)$. Thus,
\[
\mathcal{C}_{i+1} - \mathcal{C}_i = \mathcal{E}\big(q_i + r_i(1-q_i(y_{i+1}))\big) - \mathcal{E}(q_i) = 2\big(1-q_i(y_{i+1})\big)\mathcal{E}(q_i,r_i) + \big(1-q_i(y_{i+1})\big)^2\mathcal{E}(r_i,r_i)
\]
by expanding the quadratic form. Since $q_i$ is harmonic away from $Y_i \cup \Pi$ and $r_i \equiv 0$ on this set, the first term vanishes. Moreover, $r_i$ is harmonic away from $Y_i \cup \Pi \cup \{y_{i+1}\}$ and equals $1$ at $y_{i+1}$, so, by one-step decomposition,
\[
\mathcal{E}(r_i) = \sum_{z \sim y_{i+1}} c(y_{i+1},z)\big(1 - r_i(z)\big) = \mathbb{P}^X_{y_{i+1}}\big(T_{Y_i \cup \Pi} < T_{y_{i+1}}^+\big) \geq \mathbb{P}^X_{y_{i+1}}\big(\forall t \geq 1 : X_t \neq y_{i+1}\big) \geq \varepsilon,
\]
since the walk a.s. leaves any finite set and it must do so through its vertex boundary. Thus
\[
\mathcal{C}_{i+1} - \mathcal{C}_i \geq \big(1 - q_i(y_{i+1})\big)^2 \varepsilon.
\]
Finally, $y_{i+1} \notin Y_i$, so $q_i(y_{i+1}) < \tfrac{1}{2}$, hence
\[
\mathcal{C}_{i+1} - \mathcal{C}_i \geq \frac{\varepsilon}{4}.
\]
On the other hand, by \eqref{eq:cond_def} and symmetry of effective conductance\footnote{This can be seen by writing $\mathcal{C}(\Pi,Y_i) = \mathcal{E}(1-q_i) = \mathcal{E}(q_i) = \mathcal{C}(Y_i,\Pi)$.}, $\mathcal{C}_J \leq |\Pi|$. Also,
\[
\mathcal{C}_1 = \sum_{y \sim o}c(o,y) \mathbb{P}^X_y(T_\Pi < T_o) = \mathbb{P}^X_o(T_\Pi < T_o^+) \geq \varepsilon.
\]
Using the recursion, $\mathcal{C}_J - \mathcal{C}_1 \geq (J-1)\varepsilon/4$, so
\[
|\Pi| \geq \mathcal{C}_J \geq (J-1)\frac{\varepsilon}{4} + \varepsilon = \frac{\varepsilon}{4}(J+3),
\]
whence
\[
J \leq \frac{4|\Pi|}{\varepsilon} - 3 < \frac{4}{\varepsilon}|\Pi|.
\]
\end{proof}

The anticipation of this result was somewhat motivated by the proof of \cite[Proposition~3]{easoCountingMinimalCutsets2025} involving a `chained sequence' of vertices and, in particular, the second and third properties they prove, being that the sequence is uniformly well-connected to the remainder of the graph and has size controlled linearly in the size of the graph's boundary.

Note that Proposition~\ref{harmonic_cover} generalises naturally to any random walk with symmetric conductances such that for all vertices $x\in V$, $\pi(x)\leq 1$. However, we do not require this generalisation and thus we have stated the proposition in terms of the random walk $(X_t)_{t\in \mathbb{N}_0}$.

\section{Proof of Theorem~\ref{thm2_extend}}\label{sec:proof2}

Let $G$ be an infinite, connected, locally finite graph such that for all vertices $v\in V$,
\[\mathbb{P}^X_v(\forall t \geq 1: X_t\neq v)\geq \varepsilon.
\]

We follow a similar idea to the proof sketched in Section~9 of \cite{easoCountingMinimalCutsets2025}. Let $\varphi\in \mathbb{R}^V$, with law $\mathbb{P}$ and corresponding expectation $\mathbb{E}$, be the (centred) Gaussian Free Field (GFF) on $G$. Note that throughout this proof, every GFF is implicitly associated to the (lazy) random walk $(X_t)_{t\in \mathbb{N}_0}$. By definition of the GFF, for $v\in V$,
\begin{equation}\label{gff_var_ub}
\mathrm{Var}(\varphi(v)) = G(v,v) = \mathbb{E}^X_v\left( \sum_{n=0}^\infty \boldsymbol{1}_{\{X_n=v\}}\right) = \frac{1}{\mathbb{P}^X_v(\forall t \geq 1: X_t\neq v)} \leq \frac{1}{\varepsilon}
\end{equation}
where the penultimate equality follows from the fact that the number of visits of the walk to its starting vertex (including the starting visit) is geometric with `success' probability being the escape probability.

Now, fix $o\in V$ and let $C$ be the cluster of $o$ in the percolation model induced by the excursion set $\{\varphi \geq 0\}:= \{ x\in V: \varphi(x) \geq 0\}$. Fix a minimal vertex cutset $\Pi\in \mathcal{Q}_n^\mathrm{site}(o)$ from $o$ to $\infty$ in $G$. We will show that for some positive, finite constant $K:=K(\varepsilon)$,
\begin{equation}\label{bound2}
\mathbb{P}(\partial_\infty^V C = \Pi) \geq \frac{1}{K^n}.
\end{equation}
This concludes the proof since
\[
1 \geq \sum_{\Pi \in \mathcal{Q}^\mathrm{site}_n(o)} \mathbb{P}(\partial^V_\infty C = \Pi) \stackrel{\eqref{bound2}}{\geq} |\mathcal{Q}^\mathrm{site}_n(o)|/K^n.
\]

Let $A$ be the connected component of $o$ in the graph induced by the vertex set $V\setminus\Pi$. Apply Proposition~\ref{harmonic_cover} and let $Y\subset A$ be the set containing $o$ such that, for all $v\in A$, $\mathbb{P}^X_v(\mathrm{T}_Y<\mathrm{T}_\Pi)\geq 1/2$ and $|Y|<4n/\varepsilon$. Let
\[
\mathcal{F} := \left\{\varphi(\pi) \in [-2,-1]\,\ \forall \pi \in \Pi,\, \varphi(y) \in [4,5]\,\ \forall y \in Y\right\}.
\]
We claim that there exists a constant $c:=c(\varepsilon)$ such that 
\begin{equation}\label{F_bound}
\mathbb{P}(\mathcal{F}) \geq c^n.
\end{equation}
Indeed, enumerate the at most $(4/\varepsilon +1)n$ vertices of $\Pi \cup Y$ by $z_1,\ldots,z_m$ and write the prescribed interval of $\varphi(z_j)$ as $I_j$. Write $E_j := \{\varphi(z_j)\in I_j \}$. Then by the law of total probability,
\[
\mathbb{P}(\mathcal{F}) = \mathbb{P}(E_1)\prod_{j=2}^m\mathbb{P}(E_j|E_1,\ldots, E_{j-1}).
\]
Fix $j$, let $\Lambda:=\{z_1,\ldots, z_{j-1}\}$, and let $U:=V\setminus \Lambda$. By the Markov property of the GFF, Lemma~\ref{markov}, conditionally on $\varphi(z_k)=a_{z_k} \in I_k$ for $k\in \{1, \ldots, j-1\}$, we can write
\[
\varphi = \varphi_{\Lambda}+h_\Lambda
\]
where $\varphi_{\Lambda}$ is a centred GFF on $U$ that is zero on $\Lambda$, and $h_\Lambda$ is the harmonic extension into $U$ of the values of $\varphi$ in $\Lambda$, which is given by the expression \eqref{eq:hK}.
Hence, conditionally on $(\varphi(z_1),\ldots,\varphi(z_{j-1})) = a := (a_{z_1},\ldots a_{z_{j-1}})$,
\[
\varphi(z_j) = \varphi_{\Lambda}(z_j) + h_\Lambda(z_j),
\]
and so
\[
N := \varphi(z_j) \mid \{\varphi|_\Lambda = a\} \sim \mathcal{N}\big(h_\Lambda(z_j),\, G_\Lambda(z_j,z_j)\big),
\]
since $h_\Lambda$ is deterministic, and by definition of the GFF variance.

Now, we can bound this conditional mean and variance. Indeed, on $E_1 \cap \cdots \cap E_{j-1}$, \mbox{$\varphi(z_i) \in [-2, 5]$} for all $i=1,\ldots,j-1$, so, by \eqref{eq:hK}, $h_\Lambda(z_j) \in [-2,5]$. Furthermore, recall that
\[
G_\Lambda(y,y) = \mathbb{E}^X_y\left(\sum_{t<T_\Lambda} \mathbf{1}_{\{X_t=y\}}\right),
\]
so $G_\Lambda(z_j,z_j) \geq 1$. Moreover, killing the walk at $T_\Lambda$ can only decrease the Green function, so $G_\Lambda(z_j,z_j) \leq G(z_j,z_j) \leq 1/\varepsilon$ by \eqref{gff_var_ub}. Thus,
\[
1 \leq G_\Lambda(z_j,z_j) \leq \frac{1}{\varepsilon}.
\]
Now, for either $I = [-2, -1]$ or $I = [4, 5]$, we have, for all $t \in I$, that
\[
|t - h_\Lambda(z_j)| \leq 7.
\]
Writing $\mu$ for $h_\Lambda(z_j)$ and $\sigma^2$ for $G_\Lambda(z_j,z_j)$, we also have
\[
\frac{1}{\sqrt{\sigma^2}} \geq \sqrt{\varepsilon} \qquad \text{and} \qquad \frac{(t-\mu)^2}{2\sigma^2} \leq \frac{49}{2},
\]
so the Gaussian density satisfies
\[
\frac{1}{\sqrt{2\pi\sigma^2}} \exp\left(-\frac{(t-\mu)^2}{2\sigma^2}\right) \geq \sqrt{\frac{\varepsilon}{2\pi}}\, e^{-49/2}.
\]
Hence
\[
\mathbb{P}(N \in I) \geq  \sqrt{\frac{\varepsilon}{2\pi}}\, e^{-49/2} =: q_\varepsilon.
\]
And so $\mathbb{P}(E_j \mid E_1,\ldots,E_{j-1}) \geq q_\varepsilon$ for all $j$, since this bound is independent of the revealed values. Therefore
\[
\mathbb{P}(\mathcal{F}) \geq q_\varepsilon^m \geq q_\varepsilon^{(4/\varepsilon+1)n}.
\]
So we can let
\[
c(\varepsilon) := q_\varepsilon^{(4/\varepsilon+1)} = \left(\sqrt{\frac{\varepsilon}{2\pi}}\, e^{-49/2}\right)^{(4/\varepsilon+1)},
\]
which proves \eqref{F_bound} as claimed.

Now, let $b:=(b_v)_{v\in \Pi \cup Y} = (\varphi(v))_{v\in \Pi \cup Y}$ be some exact values satisfying the event $\mathcal{F}$. Conditional on $B:=\varphi|_{\Pi \cup Y}=b$ and by Lemma~\ref{markov} applied with $\Lambda = \Pi \cup Y$, we can write $\varphi = \varphi_b+ h_b$ where $\varphi_b$ is a centred GFF on $V \setminus (\Pi\cup Y)$ and zero on $\Pi \cup Y$, while $h_b$ is the harmonic extension into $V\setminus(\Pi\cup Y)$ of the values of $\varphi$ in $\Pi \cup Y$. Recall that
\[
h_b(v) = \mathbb{E}^X_v\big(\varphi(X_{T_{\Pi\cup Y}})\big)
\]
on $A$. In particular, on $A$, writing $q_v := \mathbb{P}^X_v(T_Y < T_\Pi) \geq 1/2$ by Proposition~\ref{harmonic_cover}, we have
\begin{equation}\label{eq:hb_bound}
h_b(v) \geq 4q_v -2 (1-q_v) = 6q_v-2 \geq 1.
\end{equation}
The choice of intervals in the event $\mathcal{F}$ is what forces this lower bound to be positive, the necessity of which will become apparent soon.

\begin{remark}\label{strenthened_assump}
It is at this exact point that stronger assumptions than those of \cite[Theorem~2]{easoCountingMinimalCutsets2025} become necessary. In the proof sketched in \cite[Section~9]{easoCountingMinimalCutsets2025}, the authors prescribe GFF values on either endpoint of the cutset edges; this makes the harmonic extension on $A$ positive everywhere automatically. The analogous event for a vertex cutset would be the prescription of positive GFF values on all neighbours in $A$ of cutset vertices, which is generally degree dependent. If instead one does not prescribe these values, the harmonic extension in $A$ would be negative near its boundary, causing a future argument (namely, the positivity of $g(\varepsilon)$ in \eqref{eq:sgn_positive}) to fail. As can be seen above, Proposition~\ref{harmonic_cover} side-steps this. 
\end{remark}

For each $\pi \in \Pi$, let $a_\pi$ be a neighbour in $A$. Then we have that, by Lemma~\ref{second_bound_pres} and writing $\mathbb{P}_b$ for $\mathbb{P}(\,\cdot \,|B=b)$,
\begin{align}\label{int_bound}
\mathbb{P}\big(\partial_\infty^V C = \Pi\big) &\geq \mathbb{P}\!\left(\mathcal{F},\bigcap_{y\in Y} \left\{o \stackrel{\{\varphi\geq0\}\cap A}{\longleftrightarrow} y\right\},\; \bigcap_{\pi\in\Pi} \left\{a_\pi \stackrel{\{\varphi\geq0\}\cap A}{\longleftrightarrow} Y\right\}\right)  \notag\\
&= \int_\mathcal{R}\, \mathbb{P}_b\!\left(\bigcap_{y\in Y} \left\{o \stackrel{\{\varphi\geq0\}\cap A}{\longleftrightarrow} y\right\},\; \bigcap_{\pi\in\Pi} \left\{a_\pi \stackrel{\{\varphi\geq0\}\cap A}{\longleftrightarrow} Y\right\}\right) \mathbb{P}(B \in db) \notag\\
&\geq c(\varepsilon)^n\, \inf_{b\in \mathcal{R}} \mathbb{P}_b\!\left(\bigcap_{y\in Y} \left\{o \stackrel{\{\varphi\geq0\}\cap A}{\longleftrightarrow} y\right\}\right) \inf_{b\in \mathcal{R}} \mathbb{P}_b\!\left(\bigcap_{\pi\in\Pi} \left\{a_\pi \stackrel{\{\varphi\geq0\}\cap A}{\longleftrightarrow} Y\right\}\right),
\end{align}
where $\mathcal{R}$ is the rectangle defining $\mathcal{F}$, and where the second inequality follows from \eqref{F_bound} and the fact that the law of the GFF $\varphi_b$ (in $A$), and thus of $\varphi\mid\{B=b\}$ (in $A$), is positively associated (i.e the FKG inequality).

To lower bound both remaining probabilities, it is enough to prove that, for all $v \in A$,
\[
\mathbb{P}_b\big(v \stackrel{\{\varphi\geq0\}\cap A}{\longleftrightarrow} Y\big) \geq g(\varepsilon)
\]
for some positive function $g$, since then we may apply Proposition~\ref{prop3} and the FKG inequality to the second and third term respectively. Let $S$ be the union of all clusters of $\{\varphi\geq0\}\cap A$ intersecting $Y$. That is,
\[
S = \left\{v \in A : v \stackrel{\{\varphi\geq0\}\cap A}{\longleftrightarrow} Y\right\}.
\]
Let $\bar S := S \cup \Pi \cup Y \cup D$, where $D:= \{v \in A\setminus S : v \sim S\}$, be the closure of $S$ in $A$ along with the deterministic sets $\Pi$ and $Y$, whose values are already revealed by conditioning on $B=b$. Then $\bar S$ is a stopping set, since one can first reveal the values on $\Pi$ and $Y$, then explore breadth-first from $Y$ (in $A$, because $S$ and $D$ are confined to $A$): for every unexplored neighbour in $A$ of the current explored cluster, if its value is non-negative, add it to $S$. If its value is negative, leave it in the revealed external vertex boundary $D$. This exploration determines $\bar S$ using only the GFF values on $\bar S$.

Let $v \in A$. On $\{v \in S\}$, $\mathrm{sgn}(\varphi(v)) = 1$; on $\{v \notin S\}$, either $v$ is a revealed negative vertex in $D$, or it belongs to an unexplored (graph) component $U(v)$\footnote{By `unexplored graph component', we mean the set of unexplored vertices connected to $v$ in the graph, independent of any field values. This is not to be confused with the usual meaning of a percolation component, which depends on whether the vertices are `open' or `closed.'} that is bounded by negative vertices (in $D$ or $\Pi$). In the latter case, applying Lemma~\ref{strong_markov} with $\Lambda = \bar S$ gives that, conditionally on $\mathcal{H} = \sigma\big(\bar S, \{\varphi(x) : x \in \bar S\}\big)$, $\varphi = \varphi_{\bar S} + h_{\bar S}$ with the usual definitions of each term. Since $U(v)$ is finite and its boundary is contained in $\bar S$, the walk $(X_t)_{t\in \mathbb{N}_0}$ started at $v$ first enters $\bar S$ in $\partial^VU(v)$, so \eqref{eq:hK} gives that $h_{\bar S}(v)<0$. 

Hence for $v\in A$, we may combine the cases where $v$ is revealed or not, so that
\[
\mathbb{E}_b\big(\mathrm{sgn}(\varphi(v)) \mathbf{1}_{\{v\notin S\}}\mid \mathcal{H}\big) \leq 0.
\]
Thus,
\begin{align*}
\mathbb{E}_b\big(\mathrm{sgn}(\varphi(v))\big) &= \mathbb{E}_b\big(\mathrm{sgn}(\varphi(v))\mathbf{1}_{\{v\in S\}}\big) + \mathbb{E}_b\big(\mathrm{sgn}(\varphi(v))\mathbf{1}_{\{v\notin S\}}\big) \\
&\leq \mathbb{P}_b(v \in S) = \mathbb{P}_b\big(v \stackrel{\{\varphi\geq0\}\cap A}{\longleftrightarrow} Y\big).
\end{align*}
So, it suffices to show that $\mathbb{E}_b\big(\mathrm{sgn}(\varphi(v))\big)$ is uniformly positive. Indeed, \mbox{$\varphi(v) \mid \{B=b\} \sim \mathcal{N}(\mu, \sigma^2)$} where $\mu = h_b(v) \geq 1$ from \eqref{eq:hb_bound} and $\sigma^2 \leq G(v,v) \leq 1/\varepsilon$ from \eqref{gff_var_ub}. Hence
\begin{align}\label{eq:sgn_positive}
\mathbb{E}_b\big(\mathrm{sgn}(\varphi(v))\big) &= \mathbb{P}_b(\varphi(v)>0) - \mathbb{P}_b(\varphi(v)<0) \notag\\
&= 1 - 2\,\mathbb{P}_b(\varphi(v)<0) \geq 1 - 2\Phi(-\sqrt{\varepsilon}) = 2\Phi(\sqrt{\varepsilon}) - 1 =: g(\varepsilon) > 0 
\end{align}
and this uniform positivity is due to the positivity of $h_b(v)$. Now, by positive association of $\mathbb{P}_b$, it follows from Proposition~\ref{prop3} that
\[
\mathbb{P}_b\!\left(\bigcap_{y \in Y} \left\{o \stackrel{\{\varphi\geq0\}\cap A}{\longleftrightarrow} y\right\}\right) \geq p(\varepsilon)^n,
\]
where in the context of this proposition, we are thinking of an edge being open if both its endpoints have nonnegative field (which happens with probability at least 1/4 by the FKG inequality), and where
\[
p(\varepsilon) := \left(\frac{2\Phi(\sqrt{\varepsilon}) - 1}{8}\right)^{\frac{12}{\varepsilon\left(2\Phi(\sqrt{\varepsilon})-1\right)}}.
\]
Finally, by \eqref{int_bound} and FKG, we have that
\[
\mathbb{P}\big(\partial_\infty^V C = \Pi\big) \geq c(\varepsilon)^n\, p(\varepsilon)^n\, g(\varepsilon)^n = \frac{1}{K^n}
\]
where
\[
K := \left(\left(\sqrt{\frac{\varepsilon}{2\pi}}\, e^{-49/2}\right)^{\left(\frac{4}{\varepsilon}+1\right)} \left(\frac{2\Phi(\sqrt{\varepsilon}) - 1}{8}\right)^{\frac{12}{\varepsilon\left(2\Phi(\sqrt{\varepsilon})-1\right)}} \left(2\Phi(\sqrt{\varepsilon}) - 1\right)\right)^{-1},
\]
which concludes the proof of \eqref{bound2} and thus of the theorem.

\section{Proof of Corollary~\ref{gapat1} - The Gap at 1}\label{sec:proof3}

Let $G=(V,E)$ be an infinite, connected, locally finite, vertex-transitive graph. It is well-known for such graphs that recurrence implies quadratic volume growth \cite{tesseraSharpRelationsVolume2024}, in the sense that $|B(o,n)|=O(n^2)$ for all $o\in V$ where $B(o,n)$ denotes the graph-distance ball of radius $n$ around $o$.

Suppose first that $G$ is indeed recurrent. If $G$ has linear volume growth, then the critical bond percolation parameter is equal to $1$, for example by \cite[Theorem~1.7]{hutchcroftNontrivialityPhaseTransition2024}. By the classical comparison inequality $p_c^{\mathrm{bond}} \leq p_c^{\mathrm{site}}$, the critical site percolation parameter is also 1. If instead $G$ has super-linear (yet still quadratic) volume growth, we may follow the proof of \cite[Theorem~1.7]{hutchcroftNontrivialityPhaseTransition2024}; in particular, we follow the case where \mbox{$m:= \sup\{n\in \mathbb{N}: |B(o,n)| \geq n^5 \}<\infty$}. 

Indeed, by doing so, we may see that there exist groups $H<\mathrm{Aut}(G)$ and $\Gamma<\mathrm{Aut}(G/H)$ such that: $\Gamma$ acts transitively on $V/H$; there exists a finite symmetric generating set $S$ for $\Gamma$; the growth of $\mathrm{Cay}(\Gamma, S)$ is superlinear; and $G/H$ is isomorphic to a quotient of $\mathrm{Cay}(\Gamma, S)$ by a subgroup $F$ of $\mathrm{Aut}(\mathrm{Cay}(\Gamma, S))$ of order at most $M$ - a universal constant\footnote{The constant $M$ depends only on the exponent $5$ chosen in the above definition of $m$ \cite{hutchcroftNontrivialityPhaseTransition2024}, which is fine, since $m$ is indeed finite in the case that $G$ is recurrent.}. Moreover, by Theorem~\ref{cayley_gap}, there exists a universal constant $\varepsilon>0$ such that $p_c(\mathrm{Cay}(\Gamma, S))\leq1-\varepsilon$. Hence, since isomorphisms preserve percolation parameters and by Lemma~\ref{groups}, we have that
\[
p_c(G/H) = p_c(\mathrm{Cay}(\Gamma, S)/F) \leq 1-\varepsilon_0
\]
for some universal constant $\varepsilon_0>0$ depending only on $M$ and $\varepsilon$. Indeed, let $V':=V(\mathrm{Cay}(\Gamma, S))$ and let 
\[
A_\infty := \left\{ \xi \in \{0,1\}^{V'/F}: |\{u\in V'/F: \xi(u)=1\}|=\infty \right\}.
\]
Then $A_\infty$ is increasing and measurable, and using the same notation as in Lemma~\ref{groups},
\[
\{C_u\in A_\infty\} := \{\mathbf{1}_{C_u}\in A_\infty\} =\{|C_u|=\infty\}.
\]

Choose $p>1-\varepsilon\geq p_c(\mathrm{Cay}(\Gamma, S))$ and $v\in V'$ such that $\mathbb{P}_p^{\mathrm{Cay}(\Gamma, S)}(|C_v|=\infty)>0$. Since $F$ has order at most $M$, and writing $\tau$ for the quotient map $\tau:\mathrm{Cay}(\Gamma, S)\rightarrow \mathrm{Cay}(\Gamma, S)/F$, we have that $|C_v|=\infty$ implies $|\tau(C_v)| = \infty$. Hence, by Lemma~\ref{groups} applied with $A=A_\infty$,
\[
\mathbb{P}^{\mathrm{Cay}(\Gamma, S)/F}_{1-(1-p)^M}\!\left(|C_{\tau(v)}| =\infty\right) \geq 
\mathbb{P}^{\mathrm{Cay}(\Gamma, S)}_{p}\!\left(|\tau(C_v)| = \infty\right) \geq 
\mathbb{P}^{\mathrm{Cay}(\Gamma, S)}_{p}\!\left(|C_v| = \infty\right) > 0,
\]
which implies that $p_c(\mathrm{Cay}(\Gamma, S)/F)\leq1-\varepsilon^M$ by letting $p\downarrow 1-\varepsilon$, so we may take $\varepsilon_0:=\varepsilon^M$. Finally, Theorem~\ref{benj} implies that $p_c(G) \leq 1-\varepsilon_0$, concluding the case where $G$ is recurrent.

Suppose now that $G$ is transient. Then, by \cite[Corollary~1.3]{tesseraSharpRelationsVolume2024}, there exists a universal constant $c>0$ such that the simple random walk on $G$ has escape probability at least $c$. Since $G$ is vertex-transitive, it has constant degree and thus the walk $(X_t)_{t\in \mathbb{N}_0}$ in Theorem~\ref{thm2_extend} reduces exactly to the simple random walk. Hence, we may follow the proof of Theorem~\ref{thm2_extend} with $\varepsilon=c$, so that $\kappa^{\mathrm{site}}(G) \leq K(c)=:K$, which is also universal. Moreover, for a vertex $u\in V$, the classical Peierls argument gives that
\[
\mathbb{P}_p(|C_u|<\infty) = 1-p+ p\,\mathbb{P}_p(|C_u|<\infty \mid u \text{ is open}) \leq 1-p+p\sum_{n\geq 1}q_n^{\mathrm{site}}(1-p)^n,
\]
by a union bound since, on the event that $u$ is open and its cluster is finite, $\partial^VC_u$ is closed and contains a minimal vertex cutset. Hence,
\[
\mathbb{P}_p(|C_u|<\infty) \leq 1-p+p\sum_{n\geq 1}K^n(1-p)^n,
\]
and so
\[
\theta_u(p) = \mathbb{P}_p(|C_u|=\infty) \geq p-p\sum_{n\geq 1}K^n(1-p)^n,
\]
which is positive if $p>1-1/(2K)$. Hence, $p_c \leq 1-1/(2K)$. Thus, we may take $\varepsilon_1 = \min\{\varepsilon_0, 1/(2K)\}$, concluding the proof.

\section{Proof of Corollary~\ref{dim2}}\label{sec:proof4}

If $\mathrm{Dim}(G;c) > 2$, then by properties of the supremum, there exists $d>2$ and $\beta\in(0,1)$ such that 
\begin{equation}\label{dim_bound}
|\partial^E K|_c \geq \beta |K|_\pi^{(d-1)/d},
\end{equation}
for all finite $K\subset V$. Hence, the proof of Corollary~\ref{dim2} is complete if we can show that there exists $M:=M(\beta,d)$ such that $\kappa^{\mathrm{site}}(G) \leq M < \infty$; we do so via Theorem~\ref{lp6.41}. Firstly, taking $K$ to be a singleton $\{x\}$ in \eqref{dim_bound}, we see that
\[
\pi(x) = |\partial^E \{x\}|_c \geq \beta |\{x\}|_\pi^{(d-1)/d} = \beta \pi(x)^{(d-1)/d}
\]
and since $\pi(x)>0$ for all $x$ due to connectivity of $G$, we may divide both sides by $\pi(x)^{(d-1)/d}$ so that $\pi(x)\geq \beta^d$. Hence, $|V(G)|_\pi \geq \beta^d|V(G)| = \infty$.

Now, recall the isoperimetric profile(s) defined in Section~\ref{exp_bound_con_ineq}; since the anchored profile takes an infimum over a smaller class of sets than the unanchored profile, we immediately get that $\varphi(G, A, t;c) \geq \varphi(G, t;c) \geq \beta t^{(d-1)/d} =: f(t)$ for any non-empty, finite set $A\subset V$ and for all $t>0$. Note that this bound is independent of the anchor set $A$, and, in particular, we may take $A$ to be any singleton $\{a\}$.

Furthermore, $f$ is clearly positive and increasing in $t$ since $(d-1)/d > 0$. In addition,
\[
f(t)\leq t \iff \beta t^{(d-1)/d} \leq t \iff \beta \leq t^{1/d} \impliedby t\geq \pi(a) \geq \beta^d,
\]
where the implication follows from the relevant domain of $f$ being $[\pi(a),\infty)$ and the lower bound on $\pi(a)$ from above. Now, $f(2t) = 2^{(d-1)/d} f(t)$, so if we set $\alpha := 2^{(d-1)/d}$, Theorem~\ref{lp6.41} applies and gives, for all vertices $a\in V$, that
\[
\mathscr{R}(a \leftrightarrow \infty) \leq \int_{\beta^d}^\infty \frac{4 \cdot 2^{2(d-1)/d}}{\beta^2 t^{2(d-1)/d}}\, dt.
\]
This integral is finite since $d>2$. Hence,
\[
\mathscr{C}(a \leftrightarrow \infty) \geq \left(\int_{\beta^d}^\infty \frac{4 \cdot 2^{2(d-1)/d}}{\beta^2 t^{2(d-1)/d}}\, dt\right)^{-1} =: \varepsilon(\beta,d)>0.
\]
By \eqref{first_step} and Theorem~\ref{thm2_extend}, there exists $M:=M(\beta,d):=K(\varepsilon(\beta,d))$ such that $\kappa^{\mathrm{site}}(G) \leq M < \infty$, which concludes the proof.

\printbibliography

\end{document}